\documentclass{amsart}

\usepackage[usenames,dvipsnames]{xcolor} 
\usepackage{color}
\definecolor{alexmcolor}{RGB}{9,6,250}
\usepackage{comment}
\usepackage{enumitem}
\usepackage{tabularx}
\usepackage[hang,flushmargin]{footmisc}
\usepackage{etoolbox}
\makeatletter
\patchcmd{\@setaddresses}{\indent}{\noindent}{}{}
\patchcmd{\@setaddresses}{\indent}{\noindent}{}{}
\patchcmd{\@setaddresses}{\indent}{\noindent}{}{}
\patchcmd{\@setaddresses}{\indent}{\noindent}{}{}
\makeatother

\usepackage{amsmath, amssymb, amsthm, amsfonts,tikz,hyperref}
\usepackage[capitalise,noabbrev]{cleveref}
\newcommand{\tg}{\mathcal{T}}
\newcommand{\dtg}{\mathcal{D}} 
\DeclareMathOperator{\irr}{\operatorname{Irr}}

\usepackage[backend=biber,style=numeric,maxbibnames=99]{biblatex}
\renewbibmacro{in:}{}

\newtheorem{theorem}{Theorem}[section]
\newtheorem{lemma}[theorem]{Lemma}

\newtheorem{proposition}[theorem]{Proposition}

\newtheorem{corollary}[theorem]{Corollary}

\newtheorem{definition}[theorem]{Definition}

\newtheorem{problem}{Problem}
\numberwithin{equation}{section}
\usepackage{subcaption}
\usepackage{multirow}

\tikzset
{
  triangulation1/.pic=
  {%
    \foreach \i in {1,...,5}
    {
    \node[style={circle,draw=black!100, fill=black!100}] (\i) at (\i*360/5+18:1) {};
    }
    \draw (1) -- (2) -- (3) -- (4) -- (5) -- (1);
    \draw (1) -- (3);
    \draw (1) -- (4);
  }
}

\tikzset
{
  triangulation2/.pic=
  {%
    \foreach \i in {1,...,5}
    {
    \node[style={circle,draw=black!100, fill=black!100}] (\i) at (\i*360/5+18:1) {};
    }
    \draw (1) -- (2) -- (3) -- (4) -- (5) -- (1);
    \draw (2) -- (4);
    \draw (1) -- (4);
  }
}

\tikzset
{
  triangulation3/.pic=
  {%
    \foreach \i in {1,...,5}
    {
    \node[style={circle,draw=black!100, fill=black!100}] (\i) at (\i*360/5+18:1) {};
    }
    \draw (1) -- (2) -- (3) -- (4) -- (5) -- (1);
    \draw (2) -- (5);
    \draw (2) -- (4);
  }
}

\tikzset
{
  triangulation4/.pic=
  {%
    \foreach \i in {1,...,5}
    {
    \node[style={circle,draw=black!100, fill=black!100}] (\i) at (\i*360/5+18:1) {};
    }
    \draw (1) -- (2) -- (3) -- (4) -- (5) -- (1);
    \draw (2) -- (5);
    \draw (3) -- (5);
  }
}

\tikzset
{
  triangulation5/.pic=
  {%
    \foreach \i in {1,...,5}
    {
    \node[style={circle,draw=black!100, fill=black!100}] (\i) at (\i*360/5+18:1) {};
    }
    \draw (1) -- (2) -- (3) -- (4) -- (5) -- (1);
    \draw (1) -- (3);
    \draw (3) -- (5);
  }
}


\begin{document}

\title[Sampling planar tanglegrams and pairs of disjoint triangulations]{Sampling planar tanglegrams and pairs of disjoint triangulations}
\author[Black]{Alexander E. Black}
\address{Department of Mathematics, University of California, Davis, CA 95616, USA}
\email{aeblack@ucdavis.edu}
\author[Liu]{Kevin Liu}
\address{Department of Mathematics, University of Washington, Seattle, WA 98125, USA}
\email{kliu15@uw.edu}
\author[McDonough]{Alex McDonough}
\address{Department of Mathematics, University of California, Davis, CA 95616, USA}
\email{amcd@ucdavis.edu}
\author[Nelson]{Garrett Nelson}
\address{Department of Mathematics, Kansas State University, Manhattan, KS 66506, USA}
\email{garrettnelson@ksu.edu}
\author[Wigal]{Michael C. Wigal}
\address{Department of Mathematics, Georgia Institute of Technology, Atlanta, GA 30318, USA}
\email{wigal@gatech.edu}
\author[Yin]{Mei Yin}
\address{Department of Mathematics, University of Denver, Denver, CO 80208, USA}
\email{mei.yin@du.edu}
\author[Yoo]{Youngho Yoo}
\address{Department of Mathematics, Texas A\&M University, College Station, TX 77840, USA}
\email{yyoo@tamu.edu}
\date{}

\maketitle

\vspace{-0.25in}
\begin{abstract}
    A tanglegram consists of two rooted binary trees and a perfect matching between their leaves, and a planar tanglegram is one that admits a layout with no crossings. We show that the problem of generating planar tanglegrams uniformly at random reduces to the corresponding problem for irreducible planar tanglegram layouts, which are known to be in bijection with pairs of disjoint triangulations of a convex polygon. We extend the flip operation on a single triangulation to a flip operation on pairs of disjoint triangulations.  Interestingly, the resulting flip graph is both connected and regular, and hence a random walk on this graph converges to the uniform distribution. We also show that the restriction of the flip graph to the pairs with a fixed triangulation in either coordinate is connected, and give diameter bounds that are near optimal. Our results furthermore yield new insight into the flip graph of triangulations of a convex $n$-gon with a geometric interpretation on the associahedron.
\end{abstract}

\section{Introduction}

A \emph{tanglegram} consists of two rooted binary trees and a perfect matching between their leaves. They initially arose in computer science and biology \cite{algorithm,charleston,reduction}. Tanglegrams are drawn in the plane using \emph{layouts} such as the ones shown in \cref{fig:introexample}. Layouts with the fewest number of crossings possible are of interest in applications such as estimating the number of horizontal gene transfers between species ~\cite{comparing}.

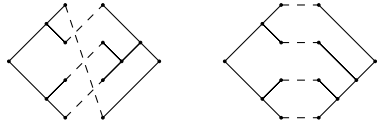
\begin{figure}[h!]
\centering 
    \begin{tikzpicture}[scale=0.5]
    \filldraw[fill=black,draw=black] (0,3) circle (1pt);
    \filldraw[fill=black,draw=black] (0,2) circle (1pt);
    \filldraw[fill=black,draw=black] (0,1) circle (1pt);
    \filldraw[fill=black,draw=black] (0,0) circle (1pt);
    \filldraw[fill=black,draw=black] (1,3) circle (1pt);
    \filldraw[fill=black,draw=black] (1,2) circle (1pt);
    \filldraw[fill=black,draw=black] (1,1) circle (1pt);
    \filldraw[fill=black,draw=black] (1,0) circle (1pt);
    \filldraw[fill=black,draw=black] (-0.5,2.5) circle (1pt);
    \filldraw[fill=black,draw=black] (-1.5,1.5) circle (1pt);
    \filldraw[fill=black,draw=black] (-0.5,0.5) circle (1pt);
    \filldraw[fill=black,draw=black] (2,2) circle (1pt);
    \filldraw[fill=black,draw=black] (1.5,1.5) circle (1pt);
    \filldraw[fill=black,draw=black] (2.5,1.5) circle (1pt);
    \draw[dashed] (0,3) -- (1,0);
    \draw[dashed] (0,2) -- (1,3);
    \draw[dashed] (0,1) -- (1,2);
    \draw[dashed] (0,0) -- (1,1);
    \draw (0,3) -- (-0.5,2.5) -- (0,2) -- (-0.5,2.5)-- (-1.5,1.5) -- (-0.5,0.5) -- (0,1) -- (-0.5,0.5) -- (0,0);
    \draw (1,3) -- (2,2) -- (1.5,1.5) -- (1,2) -- (1.5,1.5) -- (1,1) -- (2,2) -- (2.5,1.5) -- (1,0);
    \end{tikzpicture}
\qquad 
    \begin{tikzpicture}[scale=0.5]
    \filldraw[fill=black,draw=black] (0,3) circle (1pt);
    \filldraw[fill=black,draw=black] (0,2) circle (1pt);
    \filldraw[fill=black,draw=black] (0,1) circle (1pt);
    \filldraw[fill=black,draw=black] (0,0) circle (1pt);
    \filldraw[fill=black,draw=black] (1,3) circle (1pt);
    \filldraw[fill=black,draw=black] (1,2) circle (1pt);
    \filldraw[fill=black,draw=black] (1,1) circle (1pt);
    \filldraw[fill=black,draw=black] (1,0) circle (1pt);
    \filldraw[fill=black,draw=black] (-0.5,2.5) circle (1pt);
    \filldraw[fill=black,draw=black] (-1.5,1.5) circle (1pt);
    \filldraw[fill=black,draw=black] (-0.5,0.5) circle (1pt);
    \filldraw[fill=black,draw=black] (2.5,1.5) circle (1pt);
    \filldraw[fill=black,draw=black] (1.5,0.5) circle (1pt);
    \filldraw[fill=black,draw=black] (2,1) circle (1pt);
    \draw[dashed] (0,3) -- (1,3);
    \draw[dashed] (0,2) -- (1,2);
    \draw[dashed] (0,1) -- (1,1);
    \draw[dashed] (0,0) -- (1,0);
    \draw (0,3) -- (-0.5,2.5) -- (0,2) -- (-0.5,2.5)-- (-1.5,1.5) -- (-0.5,0.5) -- (0,1) -- (-0.5,0.5) -- (0,0);
    \draw (1,3) -- (2.5,1.5) -- (2,1) -- (1,2) -- (2,1) -- (1.5,0.5) -- (1,1) -- (1.5,0.5) -- (1,0);
    \end{tikzpicture}
    \caption{Two layouts for the same tanglegram.}
    \label{fig:introexample}
    \end{figure}

Combinatorial interest has grown recently, and enumerating several variations of tanglegrams has been studied \cite{count,species,countingplanar}. Algorithms for uniform sampling of tanglegrams were established in \cite{count} and \cite{fusy}, and properties of random tanglegrams were studied in~\cite{randomtanglegram}. 

A tanglegram is \emph{planar} if it has a layout with no crossings, and we refer the reader to \cite{kura,liu2022,layout,countingplanar} for many interesting results about them. The second author established a characterization of all planar layouts of a planar tanglegram in \cite{liu2022} and then proposed the problem of efficiently sampling planar tanglegrams uniformly at random.
In this paper, we consider this problem. Note that one can use the algorithms in \cite{count} and \cite{fusy} to sample tanglegrams uniformly at random until a planar tanglegram is generated, but results in \cite{count} and \cite{countingplanar} imply that planar tanglegrams are rare.

In their work enumerating planar tanglegrams, Ralaivaosaona, Ravelomanana, and Wagner \cite{countingplanar} introduced \emph{irreducible} planar tanglegrams, which are planar tanglegrams that cannot be constructed from smaller tanglegrams. We make this definition precise in \cref{preliminaries}. Let $\mathcal{P}$ and $\mathcal{I}$ respectively denote the set of planar and irreducible planar tanglegrams. \cite[Theorem 1]{countingplanar} gives a relation between the generating functions
\begin{equation}\label{THx}
\begin{split}
    T(x) & =\sum_{\mathcal{T}\in \mathcal{P}}x^{|\mathcal{T}|}, \\
H(x) & =\frac{1}{2}x^2+\sum_{\mathcal{T}\in\mathcal{I}:~|\mathcal{T}|>2}x^{|\mathcal{T}|},
\end{split}
\end{equation}
where $|\mathcal{T}|$ denotes the number of leaves in the component trees of $\mathcal{T}$. Their sequences of coefficients can be found at \cite[A257887,A349408]{oeis}. Letting $\irr(\mathcal{T})$ denote the irreducible tanglegram formed by contracting each of the smaller tanglegrams in $\mathcal{T}$ to a pair of matched leaves, we generalize these generating functions to also account for $\irr(\mathcal{T})$ by defining
\begin{equation}\label{THxy}
\begin{split}
    T(x,y) & =\sum_{\mathcal{T}\in \mathcal{P}}x^{|\mathcal{T}|}y^{|\irr(\mathcal{T})|},\\
H(x,y) & =H(xy)=\frac{1}{2}x^2y^2+\sum_{\mathcal{T}\in\mathcal{I}:~|\mathcal{T}|>2}x^{|\mathcal{T}|}y^{|\mathcal{T}|}.
\end{split}
\end{equation}
We establish a generalization of the relation on $T(x)$ and $H(x)$ from \cite[Theorem 1]{countingplanar}. Note that substituting $y=1$ recovers the original result.

\begin{theorem}\label{generatingfunction}
The following holds:
\begin{equation}\label{gfequation}
    T(x,y) = H(T(x),y)+\frac{T(x^2)y^2}{2}+xy.
\end{equation}
\end{theorem}

Using this result, we establish \cref{tanglegramalgorithm}, which reduces the problem of generating planar tanglegrams uniformly at random to computing coefficients of $T(x,y)$ and generating irreducible planar tanglegram layouts uniformly at random. In their work enumerating planar tanglegrams, Ralaivaosaona, Ravelomanana, and Wagner \cite{countingplanar} also established a natural bijection between  irreducible planar tanglegram layouts and pairs of triangulations of a convex polygon that do not share a diagonal. We call these \emph{pairs of disjoint triangulations}. Hence, one can instead consider the problem of generating pairs of disjoint triangulations uniformly at random.

Triangulations of a convex $n$-gon are one of many objects enumerated by the Catalan numbers. Sampling Catalan objects uniformly  or approximately uniformly at random has been an active area of research over the last several decades. Approaches include direct methods using properties of Catalan numbers \cite{atkinson,bacher,remy}, Boltzmann sampling with the generating function relation \cite{boltzmann}, and Markov chains on flip graphs of Catalan objects \cite{eppstein,mcshinetetali,molloy}. Enumerative results suggest that pairs of disjoint triangulations are significantly more complicated than individual triangulations. In particular, there is no known simple formula in $n$ for the number of pairs of disjoint triangulations of an $n$-gon, even after fixing one of the triangulations~\cite{ears,countingplanar}. 

We extend the flip operation on single triangulations to a new flip operation on the pairs of disjoint triangulations of an $n$-gon. We write $\mathcal D_n$ for the graph whose vertices are pairs of disjoint triangulations and whose edges are given by flips, and we call this the \emph{flip graph} on pairs of disjoint triangulations. The flip graph $\mathcal{D}_5$ is shown in \cref{D5}. Our flip operation implies many surprising properties for $\mathcal D_n$. In particular, we can define a \emph{Markov chain} $(X_t)$ from a random walk on $\mathcal D_n$. We set $X_0$ to be an arbitrary vertex and, for $t>0$, choose $X_t$ uniformly at random among the vertices adjacent to $X_{t-1}$.

\begin{theorem}\label{newmainthm}
For a fixed $n\ge 5$, let $(X_t)$ be the Markov chain defined above. Then as $t \rightarrow \infty$, $(X_t)$ converges to the uniform distribution on the vertices of $\mathcal D_n$ in total variance distance.
\end{theorem}

\begin{figure}[h!]
    \centering
    \scalebox{0.35}
    {
    \begin{tikzpicture}
        \draw (-12,0) -- (-4,6) -- (-4,2) -- (-12,0) -- (-4,-2) -- (-4,-6) -- (-12,0);
        \draw (12,0) -- (4,6) -- (4,2) -- (12,0) -- (4,-2) -- (4,-6) -- (12,0);
        \draw (-4,6) -- (4,6) -- (-4,2) -- (-4,-2) -- (4,-6) -- (-4,-6) -- (4,-2) -- (4,2) -- (-4,6);
        
        \draw[color=black, very thick,double,double distance=2pt] (4,2) -- (-4,6) to (-12,0) -- (-4,-6) -- (4,-2) -- (4,2);
        \draw[color=black, very thick,double,double distance=2pt] (-4,2) -- (4,6) to (12,0) -- (4,-6) -- (-4,-2) -- (-4,2);
        
        \filldraw[white] (-14.5,-1.25) rectangle (-9.5,1.25);
        \filldraw[white] (-6.5,4.75) rectangle (-1.5,7.25);
        \filldraw[white] (-6.5,0.75) rectangle (-1.5,3.25);
        \filldraw[white] (-6.5,-4.75) rectangle (-1.5,-7.25);
        \filldraw[white] (-6.5,-0.75) rectangle (-1.5,-3.25);
        \filldraw[white] (14.5,-1.25) rectangle (9.5,1.25);
        \filldraw[white] (6.5,4.75) rectangle (1.5,7.25);
        \filldraw[white] (6.5,0.75) rectangle (1.5,3.25);
        \filldraw[white] (6.5,-4.75) rectangle (1.5,-7.25);
        \filldraw[white] (6.5,-0.75) rectangle (1.5,-3.25);
        
        \pic at (-13.25,0) {triangulation2};
        \pic at (-10.75,0) {triangulation4};
        \pic at (-5.25,6) {triangulation1};
        \pic at (-2.75,6) {triangulation3};
        \pic at (-5.25,2) {triangulation1};
        \pic at (-2.75,2) {triangulation4}; 
        \pic at (-5.25,-2) {triangulation2};
        \pic at (-2.75,-2) {triangulation5};
        \pic at (-5.25,-6) {triangulation3};
        \pic at (-2.75,-6) {triangulation5};
        
        \pic at (13.25,0) {triangulation2};
        \pic at (10.75,0) {triangulation4};
        \pic at (5.25,6) {triangulation3};
        \pic at (2.75,6) {triangulation5};
        \pic at (5.25,2) {triangulation2};
        \pic at (2.75,2) {triangulation5};
        \pic at (5.25,-2) {triangulation1};
        \pic at (2.75,-2) {triangulation4};
        \pic at (5.25,-6) {triangulation1};
        \pic at (2.75,-6) {triangulation3};
    \end{tikzpicture}
    }
    \caption{The flip graph $\dtg_5$. Single lines indicate when only one triangulation is changed, and double lines indicate when both are changed.}
    \label{D5}
\end{figure}
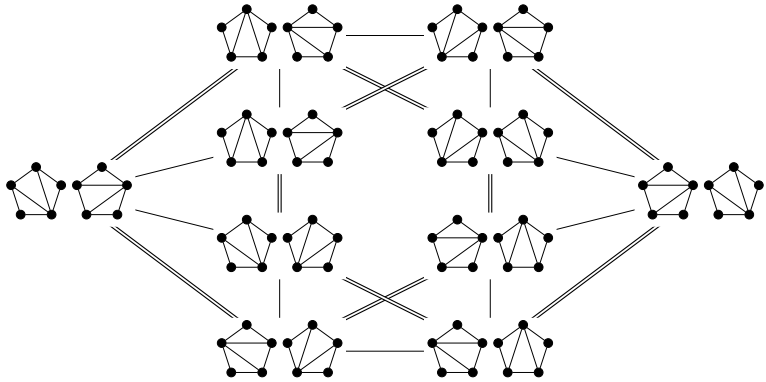

\cref{newmainthm} gives a method for near uniform sampling of pairs of disjoint triangulations. The mixing time of $X_t$ is a topic of continued study. Our general method is similar to one used by Heitsch and Tetali to study \emph{meanders}, which are also pairs of Catalan objects subject to some property~\cite{meanders}. More specifically, a {meander} is a pair of noncrossing matchings of $2n$ points that form a cycle, such as the one shown in \cref{fig:meander}. Heitsch and Tetali also constructed a Markov chain that converges to the uniform distribution, and the mixing time of their Markov chain also remains open. 


\begin{figure}[h!]
    \centering
    \begin{tikzpicture}[scale=0.6]
    \filldraw[fill=black,draw=black] (0,0) circle (1pt);
    \filldraw[fill=black,draw=black] (1,0) circle (1pt);
    \filldraw[fill=black,draw=black] (2,0) circle (1pt);
    \filldraw[fill=black,draw=black] (3,0) circle (1pt);
    \filldraw[fill=black,draw=black] (4,0) circle (1pt);
    \filldraw[fill=black,draw=black] (5,0) circle (1pt);
    \draw[color=red] (0,0) to[bend right=90] (3,0);
    \draw[color=red] (1,0) to[bend right=90] (2,0);
    \draw[color=red] (4,0) to[bend right=90] (5,0);
    \draw[color=blue] (0,0) to[bend left=90] (5,0);
    \draw[color=blue] (1,0) to[bend left=90] (4,0);
    \draw[color=blue] (2,0) to[bend left=90] (3,0);
    \end{tikzpicture}
    \caption{An example of a meander, with each noncrossing matching color-coded.}
    \label{fig:meander}
\end{figure}
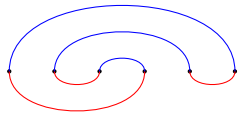

In order to prove \cref{newmainthm}, we show that $\mathcal D_n$ is regular and connected. In the process, we bound the diameter of $\mathcal D_n$. In particular, we prove the following. 

\begin{theorem}\label{mainthm1}
For any positive integer $n \geq 5$, the graph $\mathcal{D}_n$ is simple, connected, and $2(n-3)$-regular, with diameter at most $4n-16$. 
\end{theorem}

The proof of \cref{mainthm1} provides novel insight into the classical flip graph of triangulations of an $n$-gon. Namely, we show in \cref{thm:inducedconn} that the induced subgraph of triangulations of an $n$-gon disjoint from some fixed triangulation is always connected. The flip graph in this case is the set of vertices and edges of the associahedron, a polytope that plays a fundamental role in algebraic and geometric combinatorics, see \cite{ManyAssociahedra} and references therein. Our results may be rephrased within that context. The facets of the associahedron are in natural bijection with the set of diagonals of the polygon. A triangulation corresponds to a vertex of a given facet if and only if that triangulation contains the diagonal corresponding to that facet. Therefore, we have shown that for any fixed vertex $v$ of the associahedron, the induced subgraph consisting of all vertices that do not share a facet with $v$ is still connected with small diameter, a property that one can study for any polytope as discussed in \cref{op:polytope}.  
 
 This paper is organized as follows. In \cref{preliminaries}, we outline background on tanglegrams and summarize relevant results in \cite{countingplanar} on planar tanglegrams, including their connection with pairs of disjoint triangulations. In \cref{tanglegrams}, we prove \cref{generatingfunction} and apply it to construct our algorithm for sampling planar tanglegrams. In \cref{triangulations}, we define our flip graphs on pairs of disjoint triangulations and establish \cref{mainthm1,newmainthm}. We conclude in \cref{openproblems} with open problems.

\section{Preliminaries}
\label{preliminaries}

A \emph{rooted binary tree} is a tree with a distinguished vertex called the \emph{root} where each vertex has zero or two children. We consider children to be unordered, so these are different from the plane binary trees enumerated by the Catalan numbers. 

A \emph{tanglegram} $\mathcal{T}=(L,R,\sigma)$ is formed from an ordered pair of rooted binary trees $(L,R)$ with the same number of leaves and a perfect matching $\sigma$ between the leaves of $L$ and $R$. The \emph{size} of a tanglegram, denoted $|\mathcal{T}|$, is the common number of leaves in the two trees forming the tanglegram. An isomorphism between two tanglegrams $\mathcal{T}=(L,R,\sigma)$ and $\mathcal{T}'=(L',R',\sigma')$ is an isomorphism of the underlying graphs that maps $L$ to $L'$ and $R$ to $R'$.  See \cite{count,reduction} for more details. All tanglegrams in this paper are considered up to isomorphism.

Tanglegrams are drawn in the plane using \emph{layouts}, where $L$ is embedded in the plane left of the line $x=0$ with leaves on $x=0$, $R$ is embedded in the plane right of $x=1$ with leaves on $x=1$, and the matching $\sigma$ is drawn using straight lines.  A \emph{crossing} in a layout is any intersecting pair of lines induced by $\sigma$. In general, a tanglegram has multiple layouts, and a tanglegram is \emph{planar} if it has a layout in which none of the straight lines induced by $\sigma$ cross. Examples are shown in \cref{size4irreducibles}.

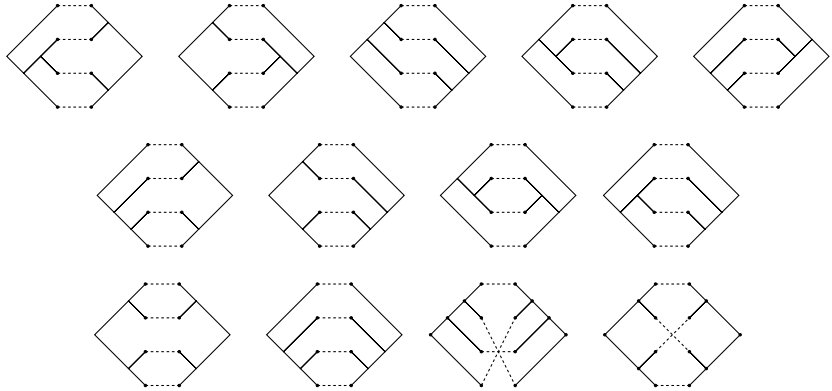
\begin{figure}[h!]
    \centering
    \begin{tikzpicture}[scale=0.45]
    \filldraw[fill=black,draw=black] (0,3) circle (1pt);
    \filldraw[fill=black,draw=black] (0,2) circle (1pt);
    \filldraw[fill=black,draw=black] (0,1) circle (1pt);
    \filldraw[fill=black,draw=black] (0,0) circle (1pt);
    \filldraw[fill=black,draw=black] (1,3) circle (1pt);
    \filldraw[fill=black,draw=black] (1,2) circle (1pt);
    \filldraw[fill=black,draw=black] (1,1) circle (1pt);
    \filldraw[fill=black,draw=black] (1,0) circle (1pt);
    \draw[dash pattern={on 1pt off 1pt}] (0,3) -- (1,3);
    \draw[dash pattern={on 1pt off 1pt}] (0,2) -- (1,2);
    \draw[dash pattern={on 1pt off 1pt}] (0,1) -- (1,1);
    \draw[dash pattern={on 1pt off 1pt}] (0,0) -- (1,0);
    \draw (0,3) -- (-1.5,1.5) -- (-1,1) -- (0,2) -- (-0.5,1.5) -- (0,1) -- (-0.5,1.5) -- (-1,1) -- (0,0);
    \draw (1,3) -- (1.5,2.5) -- (1,2) -- (1.5,2.5) -- (2.5,1.5) -- (1.5,0.5) -- (1,1) -- (1.5,0.5) -- (1,0);
    \end{tikzpicture}
    \quad 
        \begin{tikzpicture}[scale=0.45]
    \filldraw[fill=black,draw=black] (0,3) circle (1pt);
    \filldraw[fill=black,draw=black] (0,2) circle (1pt);
    \filldraw[fill=black,draw=black] (0,1) circle (1pt);
    \filldraw[fill=black,draw=black] (0,0) circle (1pt);
    \filldraw[fill=black,draw=black] (1,3) circle (1pt);
    \filldraw[fill=black,draw=black] (1,2) circle (1pt);
    \filldraw[fill=black,draw=black] (1,1) circle (1pt);
    \filldraw[fill=black,draw=black] (1,0) circle (1pt);
    \draw[dash pattern={on 1pt off 1pt}] (0,3) -- (1,3);
    \draw[dash pattern={on 1pt off 1pt}] (0,2) -- (1,2);
    \draw[dash pattern={on 1pt off 1pt}] (0,1) -- (1,1);
    \draw[dash pattern={on 1pt off 1pt}] (0,0) -- (1,0);
    \draw (0,3) -- (-0.5,2.5) -- (0,2) -- (-0.5,2.5)-- (-1.5,1.5) -- (-0.5,0.5) -- (0,1) -- (-0.5,0.5) -- (0,0);
    \draw (1,3) -- (2.5,1.5) -- (2,1) -- (1,2) -- (1.5,1.5) -- (1,1) -- (1.5,1.5) -- (2,1) -- (1,0);
    \end{tikzpicture}
        \quad 
        \begin{tikzpicture}[scale=0.45]
    \filldraw[fill=black,draw=black] (0,3) circle (1pt);
    \filldraw[fill=black,draw=black] (0,2) circle (1pt);
    \filldraw[fill=black,draw=black] (0,1) circle (1pt);
    \filldraw[fill=black,draw=black] (0,0) circle (1pt);
    \filldraw[fill=black,draw=black] (1,3) circle (1pt);
    \filldraw[fill=black,draw=black] (1,2) circle (1pt);
    \filldraw[fill=black,draw=black] (1,1) circle (1pt);
    \filldraw[fill=black,draw=black] (1,0) circle (1pt);
    \draw[dash pattern={on 1pt off 1pt}] (0,3) -- (1,3);
    \draw[dash pattern={on 1pt off 1pt}] (0,2) -- (1,2);
    \draw[dash pattern={on 1pt off 1pt}] (0,1) -- (1,1);
    \draw[dash pattern={on 1pt off 1pt}] (0,0) -- (1,0);
    \draw (0,3) -- (-0.5,2.5) -- (0,2) -- (-0.5,2.5) -- (-1,2) -- (0,1) -- (-1,2) -- (-1.5,1.5) -- (0,0);
    \draw (1,3) -- (2.5,1.5) -- (2,1) -- (1,2) -- (2,1) -- (1.5,0.5) -- (1,1) -- (1.5,0.5) -- (1,0);
    \end{tikzpicture}
            \quad 
        \begin{tikzpicture}[scale=0.45]
    \filldraw[fill=black,draw=black] (0,3) circle (1pt);
    \filldraw[fill=black,draw=black] (0,2) circle (1pt);
    \filldraw[fill=black,draw=black] (0,1) circle (1pt);
    \filldraw[fill=black,draw=black] (0,0) circle (1pt);
    \filldraw[fill=black,draw=black] (1,3) circle (1pt);
    \filldraw[fill=black,draw=black] (1,2) circle (1pt);
    \filldraw[fill=black,draw=black] (1,1) circle (1pt);
    \filldraw[fill=black,draw=black] (1,0) circle (1pt);
    \draw[dash pattern={on 1pt off 1pt}] (0,3) -- (1,3);
    \draw[dash pattern={on 1pt off 1pt}] (0,2) -- (1,2);
    \draw[dash pattern={on 1pt off 1pt}] (0,1) -- (1,1);
    \draw[dash pattern={on 1pt off 1pt}] (0,0) -- (1,0);
    \draw (0,3) -- (-1,2) -- (-0.5,1.5) -- (0,2) -- (-0.5,1.5) -- (0,1) -- (-1,2) -- (-1.5,1.5) -- (0,0);
    \draw (1,3) -- (2.5,1.5) -- (2,1) -- (1,2) -- (2,1) -- (1.5,0.5) -- (1,1) -- (1.5,0.5) -- (1,0);
    \end{tikzpicture}
    \quad 
    \begin{tikzpicture}[scale=0.45]
    \filldraw[fill=black,draw=black] (0,3) circle (1pt);
    \filldraw[fill=black,draw=black] (0,2) circle (1pt);
    \filldraw[fill=black,draw=black] (0,1) circle (1pt);
    \filldraw[fill=black,draw=black] (0,0) circle (1pt);
    \filldraw[fill=black,draw=black] (1,3) circle (1pt);
    \filldraw[fill=black,draw=black] (1,2) circle (1pt);
    \filldraw[fill=black,draw=black] (1,1) circle (1pt);
    \filldraw[fill=black,draw=black] (1,0) circle (1pt);
    \draw[dash pattern={on 1pt off 1pt}] (0,3) -- (1,3);
    \draw[dash pattern={on 1pt off 1pt}] (0,2) -- (1,2);
    \draw[dash pattern={on 1pt off 1pt}] (0,1) -- (1,1);
    \draw[dash pattern={on 1pt off 1pt}] (0,0) -- (1,0);
    \draw (0,3) -- (-1.5,1.5) -- (-1,1) -- (0,2) -- (-1,1) -- (-0.5,0.5) -- (0,1) -- (-0.5,0.5) -- (0,0);
    \draw (1,3) -- (2,2) -- (1.5,1.5) -- (1,2) -- (1.5,1.5) -- (1,1) -- (2,2) -- (2.5,1.5) -- (1,0);
    \end{tikzpicture} 
    \\
    \phantom{-}\\
    \begin{tikzpicture}[scale=0.45]
    \filldraw[fill=black,draw=black] (0,3) circle (1pt);
    \filldraw[fill=black,draw=black] (0,2) circle (1pt);
    \filldraw[fill=black,draw=black] (0,1) circle (1pt);
    \filldraw[fill=black,draw=black] (0,0) circle (1pt);
    \filldraw[fill=black,draw=black] (1,3) circle (1pt);
    \filldraw[fill=black,draw=black] (1,2) circle (1pt);
    \filldraw[fill=black,draw=black] (1,1) circle (1pt);
    \filldraw[fill=black,draw=black] (1,0) circle (1pt);
    \draw[dash pattern={on 1pt off 1pt}] (0,3) -- (1,3);
    \draw[dash pattern={on 1pt off 1pt}] (0,2) -- (1,2);
    \draw[dash pattern={on 1pt off 1pt}] (0,1) -- (1,1);
    \draw[dash pattern={on 1pt off 1pt}] (0,0) -- (1,0);
        \draw (0,3) -- (-1.5,1.5) -- (-1,1) -- (0,2) -- (-1,1) -- (-0.5,0.5) -- (0,1) -- (-0.5,0.5) -- (0,0);
 \draw (1,3) -- (1.5,2.5) -- (1,2) -- (1.5,2.5) -- (2.5,1.5) -- (1.5,0.5) -- (1,1) -- (1.5,0.5) -- (1,0);
    \end{tikzpicture}
    \quad
        \begin{tikzpicture}[scale=0.45]
    \filldraw[fill=black,draw=black] (0,3) circle (1pt);
    \filldraw[fill=black,draw=black] (0,2) circle (1pt);
    \filldraw[fill=black,draw=black] (0,1) circle (1pt);
    \filldraw[fill=black,draw=black] (0,0) circle (1pt);
    \filldraw[fill=black,draw=black] (1,3) circle (1pt);
    \filldraw[fill=black,draw=black] (1,2) circle (1pt);
    \filldraw[fill=black,draw=black] (1,1) circle (1pt);
    \filldraw[fill=black,draw=black] (1,0) circle (1pt);
    \draw[dash pattern={on 1pt off 1pt}] (0,3) -- (1,3);
    \draw[dash pattern={on 1pt off 1pt}] (0,2) -- (1,2);
    \draw[dash pattern={on 1pt off 1pt}] (0,1) -- (1,1);
    \draw[dash pattern={on 1pt off 1pt}] (0,0) -- (1,0);
    \draw (0,3) -- (-0.5,2.5) -- (0,2) -- (-0.5,2.5)-- (-1.5,1.5) -- (-0.5,0.5) -- (0,1) -- (-0.5,0.5) -- (0,0);
    \draw (1,3) -- (2.5,1.5) -- (2,1) -- (1,2) -- (2,1) -- (1.5,0.5) -- (1,1) -- (1.5,0.5) -- (1,0);
    \end{tikzpicture}
    \quad 
    \begin{tikzpicture}[scale=0.45]
    \filldraw[fill=black,draw=black] (0,3) circle (1pt);
    \filldraw[fill=black,draw=black] (0,2) circle (1pt);
    \filldraw[fill=black,draw=black] (0,1) circle (1pt);
    \filldraw[fill=black,draw=black] (0,0) circle (1pt);
    \filldraw[fill=black,draw=black] (1,3) circle (1pt);
    \filldraw[fill=black,draw=black] (1,2) circle (1pt);
    \filldraw[fill=black,draw=black] (1,1) circle (1pt);
    \filldraw[fill=black,draw=black] (1,0) circle (1pt);
    \draw[dash pattern={on 1pt off 1pt}] (0,3) -- (1,3);
    \draw[dash pattern={on 1pt off 1pt}] (0,2) -- (1,2);
    \draw[dash pattern={on 1pt off 1pt}] (0,1) -- (1,1);
    \draw[dash pattern={on 1pt off 1pt}] (0,0) -- (1,0);
    \draw (0,3) -- (-1,2) -- (-0.5,1.5) -- (0,2) -- (-0.5,1.5) -- (0,1) -- (-1,2) -- (-1.5,1.5) -- (0,0);
    \draw (1,3) -- (2.5,1.5) -- (2,1) -- (1,2) -- (1.5,1.5) -- (1,1) -- (1.5,1.5) -- (2,1) -- (1,0);
    \end{tikzpicture}\quad 
    \begin{tikzpicture}[scale=0.45]
    \filldraw[fill=black,draw=black] (0,3) circle (1pt);
    \filldraw[fill=black,draw=black] (0,2) circle (1pt);
    \filldraw[fill=black,draw=black] (0,1) circle (1pt);
    \filldraw[fill=black,draw=black] (0,0) circle (1pt);
    \filldraw[fill=black,draw=black] (1,3) circle (1pt);
    \filldraw[fill=black,draw=black] (1,2) circle (1pt);
    \filldraw[fill=black,draw=black] (1,1) circle (1pt);
    \filldraw[fill=black,draw=black] (1,0) circle (1pt);
    \draw[dash pattern={on 1pt off 1pt}] (0,3) -- (1,3);
    \draw[dash pattern={on 1pt off 1pt}] (0,2) -- (1,2);
    \draw[dash pattern={on 1pt off 1pt}] (0,1) -- (1,1);
    \draw[dash pattern={on 1pt off 1pt}] (0,0) -- (1,0);
    \draw (0,3) -- (-1.5,1.5) -- (-1,1) -- (0,2) -- (-0.5,1.5) -- (0,1) -- (-0.5,1.5) -- (-1,1) -- (0,0);
   \draw (1,3) -- (2.5,1.5) -- (2,1) -- (1,2) -- (2,1) -- (1.5,0.5) -- (1,1) -- (1.5,0.5) -- (1,0);
    \end{tikzpicture}
    \\
    \phantom{-}\\
    \begin{tikzpicture}[scale=0.45]
    \filldraw[fill=black,draw=black] (0,3) circle (1pt);
    \filldraw[fill=black,draw=black] (0,2) circle (1pt);
    \filldraw[fill=black,draw=black] (0,1) circle (1pt);
    \filldraw[fill=black,draw=black] (0,0) circle (1pt);
    \filldraw[fill=black,draw=black] (1,3) circle (1pt);
    \filldraw[fill=black,draw=black] (1,2) circle (1pt);
    \filldraw[fill=black,draw=black] (1,1) circle (1pt);
    \filldraw[fill=black,draw=black] (1,0) circle (1pt);
    \draw[dash pattern={on 1pt off 1pt}] (0,3) -- (1,3);
    \draw[dash pattern={on 1pt off 1pt}] (0,2) -- (1,2);
    \draw[dash pattern={on 1pt off 1pt}] (0,1) -- (1,1);
    \draw[dash pattern={on 1pt off 1pt}] (0,0) -- (1,0);
    \draw (0,3) -- (-0.5,2.5) -- (0,2) -- (-0.5,2.5)-- (-1.5,1.5) -- (-0.5,0.5) -- (0,1) -- (-0.5,0.5) -- (0,0);
    \draw (1,3) -- (1.5,2.5) -- (1,2) -- (1.5,2.5) -- (2.5,1.5) -- (1.5,0.5) -- (1,1) -- (1.5,0.5) -- (1,0);
    \end{tikzpicture}
    \quad 
    \begin{tikzpicture}[scale=0.45]
    \filldraw[fill=black,draw=black] (0,3) circle (1pt);
    \filldraw[fill=black,draw=black] (0,2) circle (1pt);
    \filldraw[fill=black,draw=black] (0,1) circle (1pt);
    \filldraw[fill=black,draw=black] (0,0) circle (1pt);
    \filldraw[fill=black,draw=black] (1,3) circle (1pt);
    \filldraw[fill=black,draw=black] (1,2) circle (1pt);
    \filldraw[fill=black,draw=black] (1,1) circle (1pt);
    \filldraw[fill=black,draw=black] (1,0) circle (1pt);
    \draw[dash pattern={on 1pt off 1pt}] (0,3) -- (1,3);
    \draw[dash pattern={on 1pt off 1pt}] (0,2) -- (1,2);
    \draw[dash pattern={on 1pt off 1pt}] (0,1) -- (1,1);
    \draw[dash pattern={on 1pt off 1pt}] (0,0) -- (1,0);
        \draw (0,3) -- (-1.5,1.5) -- (-1,1) -- (0,2) -- (-1,1) -- (-0.5,0.5) -- (0,1) -- (-0.5,0.5) -- (0,0);
    \draw (1,3) -- (2.5,1.5) -- (2,1) -- (1,2) -- (2,1) -- (1.5,0.5) -- (1,1) -- (1.5,0.5) -- (1,0);
    \end{tikzpicture}\quad 
    \begin{tikzpicture}[scale=0.45]
\filldraw[fill=black,draw=black] (-0,2) circle (1pt);
\filldraw[fill=black,draw=black] (-0,1) circle (1pt);
\filldraw[fill=black,draw=black] (-0,0) circle (1pt);
\filldraw[fill=black,draw=black] (-0,3) circle (1pt);
\filldraw[fill=black,draw=black] (1,2) circle (1pt);
\filldraw[fill=black,draw=black] (1,1) circle (1pt);
\filldraw[fill=black,draw=black] (1,0) circle (1pt);
\filldraw[fill=black,draw=black] (1,3) circle (1pt);
\filldraw[fill=black,draw=black] (-0.5,2.5) circle (1pt);
\filldraw[fill=black,draw=black] (-1.5,1.5) circle (1pt);
\filldraw[fill=black,draw=black] (-1,2) circle (1pt);
\filldraw[fill=black,draw=black] (1.5,2.5) circle (1pt);
\filldraw[fill=black,draw=black] (2.5,1.5) circle (1pt);
\filldraw[fill=black,draw=black] (2,2) circle (1pt);
\draw[dash pattern={on 1pt off 1pt}] (-0,2) -- (1,0);
\draw[dash pattern={on 1pt off 1pt}] (-0,1) -- (1,1);
\draw[dash pattern={on 1pt off 1pt}] (-0,0) -- (1,2);
\draw[dash pattern={on 1pt off 1pt}] (-0,3) -- (1,3);
\draw (-0,3) -- (-0.5,2.5) -- (-0,2) -- (-0.5,2.5) -- (-1,2) -- (-0,1) -- (-1,2) -- (-1.5,1.5) -- (-0,0);
\draw (1,3) -- (1.5,2.5) -- (1,2) -- (1.5,2.5) -- (2,2) -- (1,1) -- (2,2) -- (2.5,1.5) -- (1,0);
\end{tikzpicture}
\quad 
\begin{tikzpicture}[scale=0.45]
\filldraw[fill=black,draw=black] (-0,2) circle (1pt);
\filldraw[fill=black,draw=black] (-0,1) circle (1pt);
\filldraw[fill=black,draw=black] (-0,0) circle (1pt);
\filldraw[fill=black,draw=black] (-0,3) circle (1pt);
\filldraw[fill=black,draw=black] (1,2) circle (1pt);
\filldraw[fill=black,draw=black] (1,1) circle (1pt);
\filldraw[fill=black,draw=black] (1,0) circle (1pt);
\filldraw[fill=black,draw=black] (1,3) circle (1pt);
\filldraw[fill=black,draw=black] (-0.5,2.5) circle (1pt);
\filldraw[fill=black,draw=black] (-1.5,1.5) circle (1pt);
\filldraw[fill=black,draw=black] (-0.5,0.5) circle (1pt);
\filldraw[fill=black,draw=black] (1.5,2.5) circle (1pt);
\filldraw[fill=black,draw=black] (2.5,1.5) circle (1pt);
\filldraw[fill=black,draw=black] (1.5,0.5) circle (1pt);
\draw[dash pattern={on 1pt off 1pt}] (-0,2) -- (1,1);
\draw[dash pattern={on 1pt off 1pt}] (-0,1) -- (1,2);
\draw[dash pattern={on 1pt off 1pt}] (-0,0) -- (1,0);
\draw[dash pattern={on 1pt off 1pt}] (-0,3) -- (1,3);
\draw (-0,3) -- (-0.5,2.5) -- (-0,2) -- (-0.5,2.5) -- (-1.5,1.5) -- (-0.5,0.5) -- (-0,1) -- (-0.5,0.5) -- (-0,0);
\draw (1,3) -- (1.5,2.5) -- (1,2) -- (1.5,2.5) -- (2.5,1.5) -- (1.5,0.5) -- (1,1) -- (1.5,0.5) -- (1,0);
\end{tikzpicture}
    \caption{A layout for each of the $13$ tanglegrams of size four.}
    \label{size4irreducibles}
\end{figure}

 For a tanglegram $\mathcal{T}=(L,R,\sigma)$, suppose that for the internal vertices $u\in L$ and $v\in R$, the descendants of $u$ and $v$ are matched by $\sigma$. When $u$ and $v$ are not the roots of $L$ and $R$, the subtrees rooted at $v_1$ and $v_2$ with the matching induced by $\sigma$ is a \emph{proper subtanglegram} of $\mathcal{T}$.

 A tanglegram $\mathcal{T}$ is \emph{irreducible} if it contains no proper subtanglegrams. For any tanglegram $\mathcal{T}$, its irreducible component, denoted $\irr(\mathcal{T})$, is the tanglegram obtained by contracting each maximal proper subtanglegram to a pair of matched leaves. If $\mathcal{T}'$ is an irreducible tanglegram, then $\mathcal{T}'$ \emph{extends to} $\mathcal{T}$ if $\irr(\mathcal{T})=\mathcal{T}'$. Observe that in \cref{size4irreducibles}, the last two tanglegrams are irreducible and not planar. Of the eleven planar tanglegrams, five are irreducible, three have an irreducible component of size two, and three have an irreducible component of size three.

 We now state two results of \cite{countingplanar}. Note that the first result can also be derived as a consequence of \cite[Theorem 1.1]{liu2022}.

\begin{proposition}\cite[Proposition 5]{countingplanar}
\label{prop:twolayouts}
Every irreducible planar tanglegram $\mathcal{T}$ with $|\mathcal{T}|\geq 3$ has exactly two planar layouts. Moreover, the two planar layouts are mirror images of one another.
\end{proposition}


\begin{theorem}\cite[Theorem 4]{countingplanar}\label{thm:tanglegramtotriangle}
For any $n\geq 2$, there is a bijection between the following sets:
\begin{itemize}[itemsep=0mm]
    \item the set of planar layouts of irreducible planar tanglegrams of size $n$, and
    \item the set of ordered pairs of disjoint triangulations of an $(n+1)$-gon.
\end{itemize}
\end{theorem}

We describe this bijection. Starting with the two plane binary trees in a planar layout of $\mathcal{T}=(L,R,\sigma)$, draw lines from the root and all leaves to infinity. Then take the plane dual. Label the region above the root in each tree as $1$. In $L$, proceed clockwise, and in $R$, proceed counterclockwise.  An example is shown in \cref{fig:tanglegramtotriangle}. 

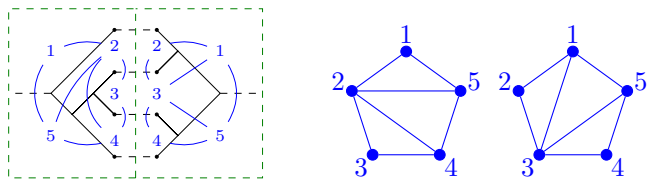
\begin{figure}[h!]
    \centering
    \scalebox{0.9}{
    \begin{tikzpicture}[scale=0.625]
    \filldraw[fill=black,draw=black] (0,3) circle (1pt);
    \filldraw[fill=black,draw=black] (0,2) circle (1pt);
    \filldraw[fill=black,draw=black] (0,1) circle (1pt);
    \filldraw[fill=black,draw=black] (0,0) circle (1pt);
    \filldraw[fill=black,draw=black] (1,3) circle (1pt);
    \filldraw[fill=black,draw=black] (1,2) circle (1pt);
    \filldraw[fill=black,draw=black] (1,1) circle (1pt);
    \filldraw[fill=black,draw=black] (1,0) circle (1pt);
    \draw[dashed] (0,3) -- (1,3);
    \draw[dashed] (0,2) -- (1,2);
    \draw[dashed] (0,1) -- (1,1);
    \draw[dashed] (0,0) -- (1,0);
    \draw[dashed] (-1.5,1.5) -- (-2.5,1.5);
    \draw[dashed] (2.5,1.5) -- (3.5,1.5);
    \draw (0,3) -- (-1.5,1.5) -- (-1,1) -- (0,2) -- (-0.5,1.5) -- (0,1) -- (-0.5,1.5) -- (-1,1) -- (0,0);
    \draw (1,3) -- (1.5,2.5) -- (1,2) -- (1.5,2.5) -- (2.5,1.5) -- (1.5,0.5) -- (1,1) -- (1.5,0.5) -- (1,0);
    \draw[color=green!50!black,dashed] (0.5,-0.5) -- (0.5,3.5) -- (-2.5,3.5) -- (-2.5,-0.5) -- (3.5,-0.5) -- (3.5,3.5) -- (0.5,3.5);
    \node[color=blue] (11) at (-1.5,2.5) {\scriptsize{$1$}};
    \node[color=blue] (12) at (0,2.625) {\scriptsize{$2$}};
    \node[color=blue] (13) at (0,1.5) {\scriptsize{$3$}};
    \node[color=blue] (14) at (0,0.375) {\scriptsize{$4$}};
    \node[color=blue] (15) at (-1.5,0.5) {\scriptsize{$5$}};
    \draw[color=blue] (11) to[bend left=15] (12);
    \draw[color=blue] (12) to[bend left=30] (13);
    \draw[color=blue] (13) to[bend left=30] (14);
    \draw[color=blue] (14) to[bend left=15] (15);
    \draw[color=blue] (15) to[bend left = 30] (11);
    \draw[color=blue] (15) to[bend left=15] (12);
    \draw[color=blue] (12) to[bend right=45] (14);
    \node[color=blue] (21) at (2.5,2.5) {\scriptsize{$1$}};
    \node[color=blue] (22) at (1,2.625) {\scriptsize{$2$}};
    \node[color=blue] (23) at (1,1.5) {\scriptsize{$3$}};
    \node[color=blue] (24) at (1,0.375) {\scriptsize{$4$}};
    \node[color=blue] (25) at (2.5,0.5) {\scriptsize{$5$}};
    \draw[color=blue] (21) to[bend right=15] (22);
    \draw[color=blue] (22) to[bend right=30] (23);
    \draw[color=blue] (23) to[bend right=30] (24);
    \draw[color=blue] (24) to[bend right=15] (25);
    \draw[color=blue] (25) to[bend right = 30] (21);
    \draw[color=blue] (21) -- (23) -- (25);
    \end{tikzpicture}
    \qquad 
    \scalebox{1.125}{
    \begin{tikzpicture}
    [every node/.style={circle, draw=blue!100, fill=blue!100, inner sep=0pt, minimum size=4pt},
    every edge/.style={draw, blue!100, thick},
    scale=.75]
    \foreach \i in {1,...,5}
    {
    \node[label={[blue]\i*360/5+18:\i},blue] (\i) at (\i*360/5+18:1) {};
    }
    \draw[color=blue] (5) -- (1) -- (2) -- (3) -- (4)-- (5);
    \draw[color=blue] (5) -- (2) -- (4);
    \end{tikzpicture}
    \begin{tikzpicture}
    [every node/.style={circle, draw=blue!100, fill=blue!100, inner sep=0pt, minimum size=4pt},
    every edge/.style={draw, blue!100, thick},
    scale=.75]
    \foreach \i in {1,...,5}
    {
    \node[label={[blue]\i*360/5+18:\i},blue] (\i) at (\i*360/5+18:1) {};
    }
    \draw[color=blue] (5) -- (1) -- (2) -- (3) -- (4)-- (5);
    \draw[color=blue] (1) -- (3) -- (5);
    \end{tikzpicture}}
    }
    \caption{The plane dual bijection from \cref{thm:tanglegramtotriangle}.}
    \label{fig:tanglegramtotriangle}
\end{figure}

\section{Planar tanglegrams}
\label{tanglegrams}

In this section, we consider the problem of uniformly sampling planar tanglegrams. We first establish \cref{generatingfunction} and its consequences. We then apply these results to reduce the problem of uniformly sampling planar tanglegrams to the corresponding problem for irreducible planar tanglegrams or their layouts in \cref{tanglegramalgorithm}.

Recall the generating functions defined in \cref{THx,THxy}. We now prove \cref{generatingfunction}, which generalizes part of \cite[Theorem 1]{countingplanar}.

\begin{proof}[Proof of \cref{generatingfunction}]
We rewrite the right hand side of \cref{gfequation} as 
\begin{equation}
    \left(H(T(x),y)-\frac{T(x)^2y^2}{2}\right)+\left(\frac{T(x)^2y^2+T(x^2)y^2}{2}\right)+xy
\end{equation}
The third summand $xy$ accounts for the unique tanglegram where both trees consist of a single vertex. For the remaining summands, we consider those with irreducible component size two and greater than two separately.

To interpret the first summand, observe that $H(x,y)-\frac{x^2y^2}{2}$ counts irreducible tanglegrams of size at least $3$, where each term $x^ky^k$ corresponds to an irreducible planar tanglegram of size $k$.
\cref{prop:twolayouts} implies that irreducible tanglegrams of size at least $3$ 
have no symmetry. Consequently, given an irreducible tanglegram $\mathcal{T}$ of size $k> 2$, fixing a layout of $\mathcal{T}$ and replacing matched leaves from top-to-bottom with planar tanglegrams $(\mathcal{T}_i)_{i=1}^k$ produces a distinct tanglegram for each selection of $(\mathcal{T}_i)_{i=1}^k$. The generating function $T(x)^k$ counts ordered $k$-tuples of planar tanglegrams, and replacing pairs of matched leaves with planar tanglegrams corresponds to replacing $x$ with $T(x)$. Hence,
$H(T(x),y)-\frac{T(x)^2y^2}{2}$ is the generating function for planar tanglegrams with $|\irr(\mathcal{T})|\geq 3$.

To interpret the second summand, tanglegrams with $|\irr(\mathcal{T})|=2$, we must
start with the unique layout for the unique tanglegram of size two and replace matched leaves with $\mathcal{T}_1$ and $\mathcal{T}_2$. However, interchanging the order of $\mathcal{T}_1$ and $\mathcal{T}_2$ produces an isomorphic tanglegram. Observe that $T(x)^2y^2$ double counts the cases when $\mathcal{T}_1\neq \mathcal{T}_2$, and $T(x^2)y^2$ counts the cases when $\mathcal{T}_1=\mathcal{T}_2$. Hence, 
$\frac{T(x)^2y^2+T(x^2)y^2}{2}$ enumerates planar tanglegrams with $|\irr(\mathcal{T})|=2$.
\end{proof}

Note that once we can efficiently compute the coefficients of $T(x)$ and $H(x)$, we can efficiently perform the composition to generate the coefficients of $T(x,y)$. Using known values of $T(x)$ and $H(x)$ from \cite[A257887,A349408]{oeis}, we give some coefficients of $T(x,y)$ in  \cref{irreducibletable}.

\begin{table}[h!]
    \centering
    \begin{tabularx}{\linewidth}{|*{9}{>{\centering\arraybackslash}X|}}
    \hline 
        $n,k$  &  2 & 3 & 4 & 5 & 6 & 7 & 8 & total \\  \hline 
        2 & 1 & & & & & & &  1\\  \hline 
        3 & 1 & 1 & & & & & &  2\\  \hline 
        4 & 3 & 3 & 5 & & & & & 11\\  \hline 
        5 & 13 & 9 & 20 & 34 & & & & 76\\  \hline 
        6 & 90 & 46 & 70 & 170 & 273 & & & 649\\  \hline 
        7 & 747 & 312 & 360 & 680 & 1638 & 2436 & & 6173\\  \hline 
        8 & 7040 & 2580 & 2435 & 3570 & 7371 & 17052 & 23391 & 63429\\ \hline 
    \end{tabularx} 
    \caption{The number of tanglegrams of size $n$ with irreducible component size $k$.}
    \label{irreducibletable}
\end{table}

We respectively use the notation $t_{n}$ and $t_{n,k}$ for the coefficient of $x^n$ in $T(x)$ and $x^ny^k$ in $T(x,y)$. We also use the notation $h_n$ for $t_{n,n}$, which is the number of irreducible planar tanglegrams of size $n$. Recall that a \emph{composition} of $n$ is a decomposition of $n$ into an ordered sum of positive integers, and we use the notation $(a_i)_{i=1}^k\models n$ to denote this. The preceding theorem and proof imply the following corollary.

\begin{corollary}\label{extension}
Let $2\leq k\leq n$. The number of ways to extend a size $k$ irreducible planar tanglegram into a size $n$ planar tanglegram, denoted $c_{n,k}$, is independent of the irreducible planar tanglegram.  Moreover,
\begin{enumerate}[label=(\alph*)]
    \item if $n$ is even, then $t_{n,2}=\frac12 \left(\sum_{i=1}^{n-1} t_it_{n-i}+t_{n/2}\right)$,
    \item if $n$ is odd, then $t_{n,2}=\frac12 \sum_{i=1}^{n-1} t_it_{n-i}$, and
    \item if $k\neq 2$, then $t_{n,k}=h_k\cdot \sum_{(a_i)_{i=1}^k\models n} t_{a_1}t_{a_2}\ldots t_{a_k}$.
\end{enumerate}
Note that $c_{n,k}=t_{n,k}/h_k$. 
\end{corollary}

Using these $c_{n,k}$ constants, we define a procedure for uniformly sampling planar tanglegrams that encodes the techniques from \cref{generatingfunction}, assuming an algorithm for uniformly sampling irreducible planar tanglegram layouts. Hence, this reduces the problem of generating planar tanglegrams to generating irreducible planar tanglegrams or their layouts.

\begin{theorem}\label{tanglegramalgorithm}
The following procedure generates a planar tanglegram of size $n\geq 3$ uniformly at random. 
\begin{enumerate}[leftmargin=*]
    \item Choose an integer $2 \leq k \leq n$ with probability $\frac{h_k c_{n,k}}{t_n}$ and generate an irreducible planar tanglegram layout $\mathcal{L}$ of size $k$ uniformly at random.
    \item 
    \begin{enumerate}[label=(\alph*),leftmargin=*]
        \item If $k\neq 2$, select $(a_i)_{i=1}^k \models n$ with probability $\frac{t_{a_1} t_{a_2}\ldots t_{a_k}}{c_{n,k}}$ and independently generate planar tanglegrams $(\mathcal{T}_i)_{i=1}^k$ of sizes $(a_i)_{i=1}^k$ uniformly at random.
    \item If $n$ is odd and $k=2$, select $(a_1,a_2)\models n$ with probability $\frac{t_{a_1} t_{a_2}}{2c_{n,2}}$ and independently generate planar tanglegrams $(\mathcal{T}_1,\mathcal{T}_2)$ of sizes $(a_1,a_2)$ uniformly at random.
    \item If $n$ is even and $k=2$, 
    \begin{itemize}[leftmargin=*]
        \item with probability $\frac{t_{n/2}}{2c_{n,2}}$, generate a single tanglegram $\mathcal{T}_1=\mathcal{T}_2$ of size $n/2$ uniformly at random, and 
        \item otherwise, select $(a_1,a_2)\models n$ with probability $\frac{t_{a_1} t_{a_2}}{\sum_{i=1}^{n-1}t_it_{n-i}}$ and independently generate planar tanglegrams $(\mathcal{T}_1,\mathcal{T}_2)$ of sizes $(a_1,a_2)$ uniformly at random.
    \end{itemize}
    \end{enumerate}
    \item In all cases, output the tanglegram corresponding to $\mathcal{L}$ with matched leaves replaced from top to bottom by $\{\mathcal{T}_i\}_{i=1}^k$.
\end{enumerate}
\end{theorem}

\begin{proof}
The definition of $c_{n,k}$ and the results of \cref{extension} imply that all of the necessary quantities in steps (1) and (2) sum to 1. We show that each tanglegram of size $n$ can be generated in two ways, and each of these possibilities has probability~$\frac{1}{2t_n}$. 


Consider a planar tanglegram $\mathcal{T}$ with $|\irr(\mathcal{T})|\geq 3$. To generate $\mathcal{T}$ in the algorithm, we must first generate one of the layouts of $\irr(\mathcal{T})$ in step (1). \cref{prop:twolayouts} implies that there are two possibilities $\mathcal{L}_1$ and $\mathcal{L}_2$, and observe that each of them has probability $\frac{h_kc_{n,k}}{t_n}\cdot \frac{1}{2h_k}=\frac{c_{n,k}}{2t_n}$ of being generated. For each $\mathcal{L}_i$, a unique list of tanglegrams $(\mathcal{T}_{i,j})_{j=1}^k$ must replace the matched leaves in $\mathcal{L}_i$ from top-to-bottom to construct $\mathcal{T}$. Letting $a_{i,j}=|\mathcal{T}_{i,j}|$, we see in that the probability $(\mathcal{T}_{i,j})_{j=1}^k$ is generated in step (2) is given by $\frac{t_{a_1}t_{a_2}\ldots t_{a_k}}{c_{n,k}}\cdot \frac{1}{t_{a_1}t_{a_2}\ldots,t_{a_k}}=\frac{1}{c_{n,k}}$. Hence, each of the two ways of generating $\mathcal{T}$ has probability $\frac{c_{n,k}}{2t_n}\cdot \frac{1}{c_{n,k}} =\frac{1}{2t_n}$ of being generated, so the probability $\mathcal{T}$ is generated is $\frac{1}{t_n}$.

Next, consider a planar tanglegram $\mathcal{T}$ with $|\irr(\mathcal{T})|=2$. This requires first generating the unique layout $\mathcal{L}$ for the unique planar tanglegram of size $2$. If $n=|\mathcal{T}|$ is odd, then two possibilities $(T_1,T_2)$ and $(T_2,T_1)$ extend $\mathcal{L}$ to $\mathcal{T}$. Letting $a_1=|\mathcal{T}_1|$ and $a_2=|\mathcal{T}_2|$, the probability of obtaining $\mathcal{T}$ is given by
\[\frac{h_2c_{n,2}}{t_n}\cdot \left(\frac{t_{a_1}t_{a_2}}{2c_{n,2}}\cdot \frac{1}{t_{a_1}t_{a_2}}+\frac{t_{a_2}t_{a_1}}{2c_{n,2}}\cdot \frac{1}{t_{a_2}t_{a_1}} \right)=\frac{1}{t_n}.\]
Note that $h_2=1$, so this term disappears in the product above. 

Now consider when $n=|\mathcal{T}|$ is even. Suppose $\mathcal{T}$ requires replacing the matched leaves in $\mathcal{L}$ with the same tanglegram $\mathcal{T}'$. With probability $\frac{t_{n/2}}{2c_{n,2}}\cdot \frac{1}{t_{n/2}}$, we generate $\mathcal{T}'$ twice in the first case of (2c), so $\mathcal{T}$ has a 
\[\frac{h_2c_{n,2}}{t_n}\cdot \frac{t_{n/2}}{2c_{n,2}}\cdot \frac{1}{t_{n/2}}=\frac{1}{2t_n}\] 
probability of being generated this way. The probability of generating $\mathcal{T}$ by generating $\mathcal{T}'$ twice in the second case of (2c) is \[\frac{h_2c_{n,2}}{t_n}\cdot \left(1-\frac{t_{n/2}}{2c_{n,2}}\cdot  \right)\cdot \frac{t_{{n/2}}t_{{n/2}}}{\sum_{i=1}^{n-1}t_it_{n-i}}\cdot \frac{1}{t_{{n/2}}t_{{n/2}}}.\]
Using \cref{extension}, $h_2=1$, and $c_{n,2}=t_{n,2}$, this simplifies to
\[\frac{c_{n,2}}{t_n}\cdot \frac{\sum_{i=1}h_ih_{n-i}}{2c_{n,2}}\cdot \frac{1}{\sum_{i=1}^{n-1}t_it_{n-i}}=\frac{1}{2t_n}.\]
The case when $\mathcal{T}$ requires replacing matched leaves in $\mathcal{L}$ with two distinct tanglegrams is done using the same properties, where we note that there is no way to generate $\mathcal{T}$ in the first case of (2c) but two ways to generate it in the second case.
\end{proof}

Note that the procedure in \cref{tanglegramalgorithm} can be applied recursively when generating $\{\mathcal{T}_i\}_{i=1}^k$, except that the tanglegrams of size 1 and 2 should be generated directly since they are unique. For efficiency reasons, one can also directly generate all planar tanglegrams below a certain size. From this, we see that uniform sampling of planar tanglegrams can be reduced to computation of the coefficients in $T(x,y)$ and  uniform sampling of irreducible planar tanglegram layouts.

\section{Pairs of disjoint triangulations}
\label{triangulations}

In this section, we consider the problem of uniformly sampling pairs of disjoint triangulations, which is equivalent to uniformly sampling irreducible planar tanglegram layouts by \cref{thm:tanglegramtotriangle}. In \cref{4.1}, we construct our flip graph $\mathcal{D}_n$ on the pairs of disjoint triangulations of an $n$-gon. We then establish \cref{mainthm1,newmainthm} using \cref{simpleregular} and \cref{thm:fixedconnected}. In \cref{4.2}, we describe the operation on irreducible planar tanglegram layouts that corresponds to flips in $\dtg_n$. 

\subsection{A flip graph on pairs of disjoint triangulations}\label{4.1}
Throughout this subsection, we fix a labeling of the convex $n$-gon using $[n]=\{1,2,\ldots, n\}$, and use pairs $(a,b)$ with $a,b\in [n]$ to denote diagonals. For a triangulation $T$, we use the notation $(a,b)\in T$ to denote that the diagonal $(a,b)$ is in the triangulation $T$. 

If $(a,b)$ is a diagonal of a triangulation $T$, then deleting $(a,b)$ creates a unique quadrilaterial $(a,a',b,b')$ for some $a',b'\in[n]$. A \emph{flip} at $(a,b) \in T$ replaces $(a,b)$ with $(a',b')$, resulting in another triangulation of the $n$-gon. We extend the flip operation to pairs of disjoint triangulations. An example is shown in \cref{fig:flip}.

\begin{definition}\label{flipdef}\normalfont
Let $(T_1,T_2)$ be an ordered pair of disjoint triangulations of an $n$-gon, and suppose $(a,b)\in T_i$ for some $i\in[2]$. A \emph{flip} at $(a,b)_i\in (T_1,T_2)$ is defined as
\begin{enumerate}[label=(\alph*)]
    \item flip $(a,b)\in T_i$, and
    \item if the resulting diagonal $(a',b')$ is in $T_j$ for $j\neq i$, then flip $(a',b')\in T_j$.
\end{enumerate}
When only (a) is performed, we refer to this as a \emph{single flip}, and when both (a) and (b) are performed, we refer to this as a \emph{double flip}.
\end{definition}

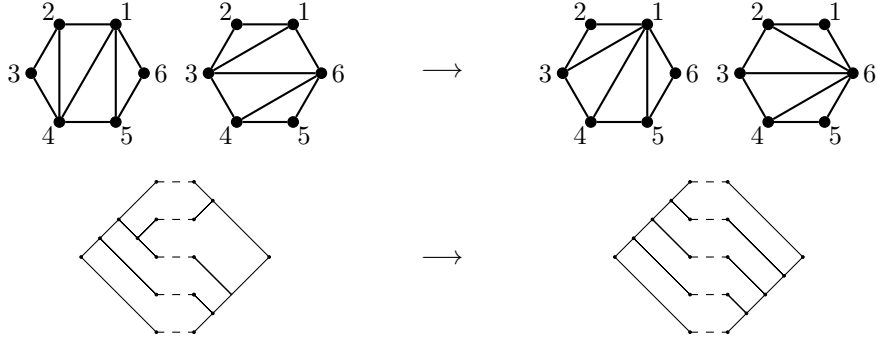
\begin{figure}[h!]
\begin{center}
\scalebox{1}{
\begin{tikzpicture}
[every node/.style={circle, draw=black!100, fill=black!100, inner sep=0pt, minimum size=4pt},
every edge/.style={draw, black!100, thick},
scale=.75]
\foreach \i in {1,...,6}
{
\node[label={\i*360/6:\i}] (\i) at (\i*360/6:1) {};
}
\draw[black!100,thick] (1) -- (2) -- (3) -- (4)-- (5)-- (6) -- (1) ;

\draw[black!100,thick] (1) -- (4);

\draw[black!100,thick] (1) -- (5);

\draw[black!100,thick] (4) -- (2);
\end{tikzpicture}
\begin{tikzpicture}
[every node/.style={circle, draw=black!100, fill=black!100, inner sep=0pt, minimum size=4pt},
every edge/.style={draw, black!100, thick},
scale=.75]
\foreach \i in {1,...,6}
{
\node[label={\i*360/6:\i}] (\i) at (\i*360/6:1) {};
}
\draw[black!100,thick] (1) -- (2) -- (3) -- (4)-- (5)-- (6) -- (1) ;

\draw[black!100,thick] (1) -- (3);

\draw[black!100,thick] (3) -- (6);

\draw[black!100,thick] (4) -- (6);
\end{tikzpicture}
\begin{tikzpicture}
\foreach \i in {1,...,6}
{
\node at (\i*360/6:1) {};
}
\node at (0,0) {$\longrightarrow$};
\end{tikzpicture}
\begin{tikzpicture}
[every node/.style={circle, draw=black!100, fill=black!100, inner sep=0pt, minimum size=4pt},
every edge/.style={draw, black!100, thick},
scale=.75]

\foreach \i in {1,...,6}
{
\node[label={\i*360/6:\i}] (\i) at (\i*360/6:1) {};
}
\draw[black!100,thick] (1) -- (2) -- (3) -- (4)-- (5)-- (6) -- (1) ;

\draw[black!100,thick] (1) -- (4);

\draw[black!100,thick] (1) -- (5);

\draw[black!100,thick] (1) -- (3);
\end{tikzpicture}
\begin{tikzpicture}
[every node/.style={circle, draw=black!100, fill=black!100, inner sep=0pt, minimum size=4pt},
every edge/.style={draw, black!100, thick},
scale=.75]
\foreach \i in {1,...,6}
{
\node[label={\i*360/6:\i}] (\i) at (\i*360/6:1) {};
}
\draw[black!100,thick] (1) -- (2) -- (3) -- (4)-- (5)-- (6) -- (1) ;

\draw[black!100,thick] (2) -- (6);

\draw[black!100,thick] (3) -- (6);

\draw[black!100,thick] (4) -- (6);
\end{tikzpicture}
}\\
\phantom{-}\\
\begin{tikzpicture}[scale=0.5]
\draw (0,4) -- (-1,3) -- (-0.5,2.5) -- (0,3) -- (-0.5,2.5) -- (0,2) -- (-1,3) -- (-1.5,2.5) -- (0,1) -- (-1.5,2.5) --  (-2,2) -- (0,0);
\draw (1,4) -- (1.5,3.5) -- (1,3) -- (1.5,3.5) -- (3,2) -- (2,1) -- (1,2) -- (2,1) -- (1.5,0.5) -- (1,1) -- (1.5,0.5) -- (1,0);
\draw[dashed] (0,4) -- (1,4);
\draw[dashed] (0,3) -- (1,3);
\draw[dashed] (0,2) -- (1,2);
\draw[dashed] (0,1) -- (1,1);
\draw[dashed] (0,0) -- (1,0);
\filldraw[fill=black,draw=black] (0,4) circle (1pt);
\filldraw[fill=black,draw=black] (0,3) circle (1pt);
\filldraw[fill=black,draw=black] (0,2) circle (1pt);
\filldraw[fill=black,draw=black] (0,1) circle (1pt);
\filldraw[fill=black,draw=black] (0,0) circle (1pt);
\filldraw[fill=black,draw=black] (1,4) circle (1pt);
\filldraw[fill=black,draw=black] (1,3) circle (1pt);
\filldraw[fill=black,draw=black] (1,2) circle (1pt);
\filldraw[fill=black,draw=black] (1,1) circle (1pt);
\filldraw[fill=black,draw=black] (1,0) circle (1pt);
\filldraw[fill=black,draw=black] (-1,3) circle (1pt);
\filldraw[fill=black,draw=black] (-0.5,2.5) circle (1pt);
\filldraw[fill=black,draw=black] (-1.5,2.5) circle (1pt);
\filldraw[fill=black,draw=black] (-2,2) circle (1pt);
\filldraw[fill=black,draw=black] (1.5,3.5) circle (1pt);
\filldraw[fill=black,draw=black] (3,2) circle (1pt);
\filldraw[fill=black,draw=black] (1.5,0.5) circle (1pt);
\filldraw[fill=black,draw=black] (1,1) circle (1pt);
\end{tikzpicture}
\qquad \qquad \quad 
\begin{tikzpicture}
\node at (0,2) {$\longrightarrow$};
\node at (0,1.125) {};
\end{tikzpicture}
\qquad \qquad \quad 
\begin{tikzpicture}[scale=0.5]
\draw (0,4) -- (-0.5,3.5) -- (0,3) -- (-0.5,3.5) -- (-1,3) -- (0,2) -- (-1,3) -- (-1.5,2.5) -- (0,1) -- (-1.5,2.5) -- (-2,2) -- (0,0);
\draw (1,4) -- (3,2) -- (2.5,1.5) -- (1,3) -- (2.5,1.5) -- (2,1) -- (1,2) -- (2,1) -- (1.5,0.5) -- (1,1) -- (1.5,0.5) -- (1,0);
\draw[dashed] (0,4) -- (1,4);
\draw[dashed] (0,3) -- (1,3);
\draw[dashed] (0,2) -- (1,2);
\draw[dashed] (0,1) -- (1,1);
\draw[dashed] (0,0) -- (1,0);
\filldraw[fill=black,draw=black] (0,4) circle (1pt);
\filldraw[fill=black,draw=black] (0,3) circle (1pt);
\filldraw[fill=black,draw=black] (0,2) circle (1pt);
\filldraw[fill=black,draw=black] (0,1) circle (1pt);
\filldraw[fill=black,draw=black] (0,0) circle (1pt);
\filldraw[fill=black,draw=black] (1,4) circle (1pt);
\filldraw[fill=black,draw=black] (1,3) circle (1pt);
\filldraw[fill=black,draw=black] (1,2) circle (1pt);
\filldraw[fill=black,draw=black] (1,1) circle (1pt);
\filldraw[fill=black,draw=black] (1,0) circle (1pt);
\filldraw[fill=black,draw=black] (-0.5,3.5) circle (1pt);
\filldraw[fill=black,draw=black] (-1,3) circle (1pt);
\filldraw[fill=black,draw=black] (-1.5,2.5) circle (1pt);
\filldraw[fill=black,draw=black] (-2,2) circle (1pt);
\filldraw[fill=black,draw=black] (1.5,0.5) circle (1pt);
\filldraw[fill=black,draw=black] (2,1) circle (1pt);
\filldraw[fill=black,draw=black] (2.5,1.5) circle (1pt);
\filldraw[fill=black,draw=black] (3,2) circle (1pt);
\end{tikzpicture}
\end{center}
\caption{A (double) flip at $(2,4)_1$ with the corresponding irreducible planar tanglegram layouts shown.}
\label{fig:flip}
\end{figure}

\begin{lemma}\label{2-flip}
Let $(T_1,T_2)$ be an ordered pair of disjoint triangulations of an $n$-gon. If $(T_1',T_2')$ is obtained from $(T_1,T_2)$ by a flip at $(a,b)_i$, then $(T_1',T_2')$ is also a pair of disjoint triangulations. Furthermore, $(T_1,T_2)$ can also be obtained from $(T_1',T_2')$ by some flip.
\end{lemma}

\begin{proof}
Without loss of generality, assume we flip the edge $(a,b)\in T_1$ to obtain $T_1'$. If after the flip, the new diagonal $(a',b')$ does not coincide with any diagonals of $T_2$, then we are done. Note that in this case, we have that $T_2=T_2'$, and the flip $(a',b')_2$ allows us to obtain $(T_1,T_2)$ from $(T_1',T_2')$.

Otherwise, we flip $(a',b')$ in the second polygon to obtain $T_2'$. The resulting diagonal $(a'',b'')$ crosses $(a',b')$, and hence cannot appear in $T_1'$. Hence, $(T_1',T_2')$ is a pair of disjoint triangulations. In this case, observe that the flip $(a'',b'')_2$ allows us to obtain $(T_1,T_2)$ from $(T_1',T_2')$.
\end{proof}

The mutual reachability between two pairs of disjoint triangulations allows us to now formally define the \emph{flip graph on pairs of disjoint triangulations}.

\begin{definition}\normalfont 
Let $\dtg_n$ denote the (undirected) graph with vertices corresponding to ordered pairs of disjoint triangulations of an $n$-gon and adjacency given by flips.
\end{definition}

Observe that $\dtg_3$ is a single vertex graph and $\dtg_4$ is a path graph on two vertices. Hence, we are primarily interested in the cases $n\geq 5$.

For a given $(T_1,T_2)\in \dtg_n$, there are $2(n-3)$ diagonals on which a flip can be performed.
It is not difficult to see that if $n\geq 5$, then flipping at different diagonals results in different pairs of disjoint triangulations.

\begin{lemma}\label{simpleregular}
For any integer $n\geq 5$, the flip graph $\dtg_n$ is simple and $2(n-3)$-regular.
\end{lemma}

\begin{proof}
It suffices to show that for any pair $(T_1,T_2)$ of disjoint triangulations, a flip at $(a,b)\in T_1$ and a flip at $(c,d)\in T_2$  cannot result in the same pair $(T_1',T_2')$. Let $(a',b') \in T_1'$ and $(c',d')\in T_2'$ be the diagonals obtained by flipping $(a,b)\in T_1$ and $(c,d)\in T_2$ respectively.
If $(T_1',T_2')$ is obtained from $(T_1,T_2)$ by flipping $(a,b)\in T_1$, then $(a',b')=(c,d)$, and if $(T_1',T_2')$ is obtained from $(T_1,T_2)$ by flipping $(c,d)\in T_2$, then $(c',d')=(a,b)$. This implies that $(a,c,b,d)$ is a quadrilateral in both $T_1$ and $T_2$. Since $n\geq 5$, this implies that $T_1$ and $T_2$ share a diagonal, which is a contradiction.
\end{proof}

Let $\tg_n$ denote the flip graph for triangulations of an $n$-gon. For any triangulation $S$ of an $n$-gon, the subgraph of $\tg_n$ induced by the set of triangulations disjoint from $S$ is denoted $\tg_n(S)$. Pournin showed in \cite{pournin2014diameter} that the diameter of $\tg_n$ is $2n - 10$ for $n > 12$. We show connectedness of $\dtg_n$ and a linear diameter bound for $\dtg_n$ by first showing corresponding statements for $\tg_n(S)$.

\begin{theorem} \label{thm:inducedconn}
Let $n\geq 5$, and let $S$ be a triangulation of the $n$-gon. Then $\tg_n(S)$ is connected, and its diameter is at most $2n-8$.
\end{theorem}

\begin{proof}
Every triangulation contains a diagonal of the form $(i,i+2)$, so we assume without loss of generality that $S$ contains the diagonal $(2,n)$. Let $\Delta$ denote the triangulation consisting of $\{(1,i):3\leq i\leq n-1\}$, which is called a \emph{standard triangulation} in Chapter $1$ of \cite{TriangulationsBible}. We show that every triangulation $T$ disjoint from $S$ is connected to $\Delta$ by a path in $\tg_n(S)$ of length at most $n-4$. In fact, we claim that if $T$ contains $d$ edges of the form $\{(1,i):3\leq i\leq n-1\}$, then it is connected to $\Delta$ by a path in $\tg_n(S)$ of length at most $n-3-d$. Note that $d\geq 1$, as $(2,n)\notin T$ implies the existence of some diagonal of the form $(1,i)$.

We prove the claim by induction on $n-3-d$. If $d=n-3$, then $T=\Delta$, and these triangulations are connected by a path of length $n-3-(n-3)=0$ as needed. Now suppose $d<n-3$, and let $3 \le i_1<i_2<\ldots<i_d\le n-1$ denote the indices such that $(1,i_j)\in T$ for all $i_j$. Since $d<n-3$, it must be that $i_{j+1}-i_j>1$ for some $j\in \{1,\ldots,d\}$. Consider the polygon with vertices $\{1,i_j,i_j+1,i_j+2,\ldots,i_{j+1}\}$ with triangulation $T_1$ induced by $T$. Observe that $T_1$ cannot contain any edges of the form $(1,i)$, and since $i_j,1,i_{j+1}$ are consecutive vertices, this can only occur if $(i_j,i_{j+1})\in T_1$. Flipping this diagonal results in $(1,i')$ for some $i_j<i'<i_{j+1}$. Note that this diagonal $(1,i')$ cannot appear in $S$ since $S$ contains $(2,n)$. Hence, flipping $(i_j,i_{j+1})$ in $T$ results in some triangulation $T'$ disjoint from $S$ that contains $d+1$ diagonals of the form $\{(1,i):3\leq i\leq n-1\}$. By the inductive hypothesis, $T'$ and $\Delta$ are connected by a path in $\tg_n(S)$ of length at most $n-3-d-1$, and hence $T$ is connected to $\Delta$ by a path in $\tg_n(S)$ of length at most $n-3-d$. 

Now let $T_1$ and $T_2$ be any two triangulations disjoint from $S$. Choosing $\Delta$ as above, we see that each $T_i$ is connected to $\Delta$ by a path of length at most $n-4$ in $\tg_n(S)$, implying $T_1$ and $T_2$ are connected by a path of length at most $2n-8$.
\end{proof}

Pournin's result for $\text{diam}(\tg_n)$ when $n>12$ combined with known values for $5\leq n\leq 12$ imply that for all $n\geq 5$, we have that $\text{diam}(\tg_n)\leq 2n-8$. Hence, the above result implies the following statements for $\dtg_n$.

\begin{corollary} \label{thm:fixedconnected}
For $n\geq 5$, the flip graph $\dtg_n$ is connected and 
\[\text{diam}(\dtg_n) \leq \text{diam}(\tg_n)+2n-8 \leq 4n-16.\]
\end{corollary}
\begin{proof}
Let $(T_1,T_2), (T_3,T_4)\in \dtg_n$.
Then there is a path of length at most $\text{diam}(\tg_n)$ in $\dtg_n$ from $(T_1,T_2)$ to $(T_3,T)$ for some $T$ disjoint from $T_3$. By Theorem \ref{thm:inducedconn}, there is a path from $(T_3,T)$ to $(T_3,T_4)$ of length at most $2n-8$ in $\dtg_n$.
\end{proof}

It is well known that a random walk on a regular connected graph converges to the uniform distribution in total variation distance if it is aperiodic. Aperiodicity follows from the fact that $\mathcal{D}_n$ contains cycles of size $3$ for all $n\ge 5$, which is shown using the standard triangulations in \cref{acyclic}. This implies that \cref{newmainthm} follows from \cref{mainthm1}, which is immediate from \cref{simpleregular} and \cref{thm:fixedconnected}. 

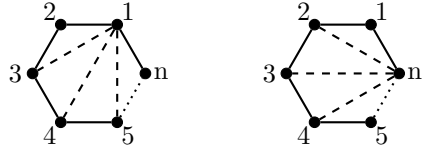
\begin{figure}[h!]
\centering
\scalebox{1}{
    \begin{tikzpicture}
[every node/.style={circle, draw=black!100, fill=black!100, inner sep=0pt, minimum size=4pt},
every edge/.style={draw, black!100, thick},
scale=.75]

\foreach \i in {1,...,5}
{
\node[label={\i*360/6:\i}] (\i) at (\i*360/6:1) {};
}
\node[label={0:n}] (n) at (0:1) {};
\draw[black!100,thick] (n) -- (1) -- (2) -- (3) -- (4)-- (5);
\draw[black!100,thick,dotted] (5) -- (n); 
\draw[black!100,thick,dashed] (1) -- (3);
\draw[black!100,thick,dashed] (1) -- (4);
\draw[black!100,thick,dashed] (1) -- (5);
\end{tikzpicture}
\quad \quad \quad
\begin{tikzpicture}
[every node/.style={circle, draw=black!100, fill=black!100, inner sep=0pt, minimum size=4pt},
every edge/.style={draw, black!100, thick},
scale=.75]

\foreach \i in {1,...,5}
{
\node[label={\i*360/6:\i}] (\i) at (\i*360/6:1) {};
}
\node[label={0:n}] (n) at (0:1) {};
\draw[black!100,thick] (n) -- (1) -- (2) -- (3) -- (4)-- (5);

\draw[black!100,thick,dotted] (5) -- (n); 
\draw[black!100,thick,dashed] (n) -- (2);
\draw[black!100,thick,dashed] (n) -- (3);
\draw[black!100,thick,dashed] (n) -- (4);
\end{tikzpicture}}
\caption{A pair of disjoint triangulations where two sequences of flips $(n,2)_2,(2,4)_1$ and $(n,2)_2$, $(1,3)_2$, $(2,4)_1$ both result in the original pair again.}
\label{acyclic}
\end{figure}

\subsection{Rotations of irreducible planar tanglegram layouts}\label{4.2}

In this subsection, we establish the operation on planar layouts of irreducible planar tanglegrams that corresponds to flips in pairs of disjoint triangulations. Our operation uses the well-known rotation on plane binary trees, which is given in \cref{fig:rotation}. 
Note that our resulting correspondence is not the first of its kind, and we refer the reader to \cite[Lemma 1]{sleator} for one such example.

\begin{figure}[h!]
    \centering
    \begin{tikzpicture}
    \draw (0,0.25) -- (0,0) -- (1,-1) -- (1.375,-1.75) -- (0.625,-1.75) -- (1,-1);
    \draw (0,0) -- (-1,-1) -- (-1.5,-1.5) -- (-1.875,-2.25) -- (-1.125,-2.25) -- (-1.5,-1.5);
    \draw (-1,-1) -- (-0.5,-1.5) -- (-0.875,-2.25) -- (-0.125,-2.25) -- (-0.5,-1.5);
    \node at (0,0.5) {$\vdots$};
    \node at (1,-1.5) {$P_3$};
    \node at (-0.5,-2) {$P_2$};
    \node at (-1.5,-2) {$P_1$};
    \filldraw[fill=black,draw=black] (0,0) circle (1pt);
    \filldraw[fill=black,draw=black] (-1,-1) circle (1pt);
    \node at (-0.25,0) {$u$};
    \node at (-1.25,-1) {$v$};
    \node[color=blue] (a) at (-1,-0.5) {$a$};
    \node[color=blue] (b) at (-1,-1.5) {$b$};
    \node[color=blue] (c) at (0,-1) {$c$};
    \node[color=blue] (d) at (1,-0.5) {$d$};
    \draw[color=blue] (a) to[bend right=90] (b);
    \draw[color=blue] (b) -- (c) -- (d);
    \draw[color=blue] (a) -- (c);
    \draw[color=blue] (a) to[bend left=90] (d);
    \end{tikzpicture}
    \quad 
    \begin{tikzpicture}
    \draw[->] (-1,0.25) to[bend left=45] (1,0.25);
    \draw[<-] (-1,-0.25) to[bend right=45] (1,-0.25);
    \node at (0,1) {rotation at $v$};
    \node at (0,-1) {rotation at $u$};
    \end{tikzpicture}
    \quad 
    \begin{tikzpicture}
    \draw (0,0.25) -- (0,0) -- (-1,-1) -- (-1.375,-1.75) -- (-0.625,-1.75) -- (-1,-1);
    \draw (0,0) -- (1,-1) -- (1.5,-1.5) -- (1.875,-2.25) -- (1.125,-2.25) -- (1.5,-1.5);
    \draw (1,-1) -- (0.5,-1.5) -- (0.875,-2.25) -- (0.125,-2.25) -- (0.5,-1.5);
    \node at (0,0.5) {$\vdots$};
    \node at (-1,-1.5) {$P_1$};
    \node at (0.5,-2) {$P_2$};
    \node at (1.5,-2) {$P_3$};
    \filldraw[fill=black,draw=black] (0,0) circle (1pt);
    \filldraw[fill=black,draw=black] (1,-1) circle (1pt);
    \node at (0.25,0) {$v$};
    \node at (1.25,-1) {$u$};
    \node[color=blue] (a) at (-1,-0.5) {$a$};
    \node[color=blue] (c) at (1,-1.5) {$c$};
    \node[color=blue] (b) at (0,-1) {$b$};
    \node[color=blue] (d) at (1,-0.5) {$d$};
    \draw[color=blue] (c) to[bend right=90] (d);
    \draw[color=blue] (b) -- (d);
    \draw[color=blue] (a) -- (b) --  (c);
    \draw[color=blue] (a) to[bend left=90] (d);
    \end{tikzpicture}
    \caption{The rotation operation on rooted binary trees. A portion of the plane duals are shown in blue.}
    \label{fig:rotation}
\end{figure}
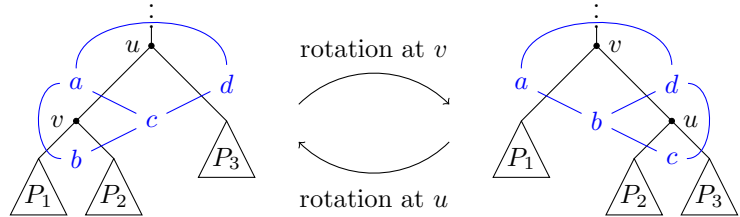

\begin{lemma}\label{onerotation}
Let $P$ be a rooted plane tree and $T$ be its plane dual triangulation. The following operations correspond:
\begin{itemize}
    \item a rotation at vertex $v\in V(P)$, and
    \item a flip at the diagonal in $T$ corresponding to the two regions separated by the edge between $v$ and its parent.
\end{itemize}
\end{lemma}

\begin{proof}
Suppose $v$ is the left child of its parent $u$. If the regions of the plane dual surrounding $v$ and $u$ are $\{a,b,c,d\}$ as in \cref{fig:rotation}, then observe that a rotation at $v$ replaces the diagonal $(a,c)$ in the dual with $(b,d)$ while preserving all remaining diagonals. This is precisely a flip at the diagonal $(a,c)$. When $v$ is the right child of its parent, then a similar argument applies. 
\end{proof}

We now generalize the rotation operation to planar layouts of irreducible planar tanglegrams. Note that any such planar layout is determined by the two component plane binary trees $P_1$ and $P_2$, as all edges between matched leaves must be horizontal. Hence, we denote a planar layout simply as $(P_1,P_2)$. Furthermore, in the plane dual bijection of \cref{thm:tanglegramtotriangle}, if two vertices $u$ and $v$ form the roots of a proper subtanglegram, then the regions adjacent to their parent edges form the shared diagonal in the dual of $(P_1,P_2)$.

\begin{definition}\label{def:rotation}\normalfont
Let $(P_1,P_2)$ be a planar layout of an irreducible planar tanglegram. A \emph{rotation} at $u\in V(P_1)\cup V(P_2)$ is defined as
\begin{enumerate}[label=(\alph*)]
    \item rotate $u\in V(P_i)$, and 
    \item if a proper subtanglegram is formed, then rotate at the vertex $v\in V(P_j)$ whose subtree forms part of a proper subtanglegram, where $j\neq i$.
\end{enumerate}
\end{definition}

\begin{theorem}
Let $\mathcal{L}_n$ be the graph on planar layouts of irreducible planar tanglegrams of size $n$ with edges given by rotations. Then $\mathcal{L}_n$ is isomorphic to $\dtg_{n+1}$.
\end{theorem}

\begin{proof}
A bijection between the vertices of $\mathcal{L}_n$ and $\dtg_{n+1}$ is given by \cref{thm:tanglegramtotriangle}, so it suffices to show that the rotation and flip operations correspond appropriately. Let $(P_1,P_2),(P_1',P_2')\in V(\mathcal{L}_n)$ with respective plane dual pairs of triangulations $(T_1,T_2),(T_1',T_2')\in V(\dtg_{n+1})$. We claim that $(P_1',P_2')$ can be obtained by some rotation in $(P_1,P_2)$ if and only if $(T_1',T_2')$ can be obtained by some flip in $(T_1,T_2)$.

Suppose $(P_1',P_2')$ can be obtained from $(P_1,P_2)$ by some rotation, which we assume without loss of generality is at some $v\in V(P_1)$. By \cref{onerotation}, we have that $T_1'$ is obtained from $T_1$ by some flip. If $(P_1',P_2)$ does not contain any proper subtanglegrams, then the result follows. Otherwise, a proper subtanglegram in $(P_1',P_2)$ corresponds to a shared diagonal in $(T_1',T_2)$, and we know that there is only one shared diagonal in $(T_1',T_2)$. Observe that there is a unique flip in $T_2$ that removes the shared diagonal in $(T_1',T_2)$. The rotation in $P_2$ described in \cref{def:rotation}(b) removes the proper subtanglegram in $(P_1',P_2)$, and hence this rotation in $P_2$ and the necessary flip in $T_2'$ must coincide. We conclude that $(T_1',T_2')$ is obtained from $(T_1,T_2)$ by some flip. A similar argument implies the converse. 
\end{proof}

\section{Open Problems}\label{openproblems}

Running the Markov chain on $\dtg_n$ from \cref{newmainthm} sufficiently many iterations allows for approximately uniform sampling of pairs of disjoint triangulations. Determining the number of iterations needed remains open, similar to the Markov chain in \cite{meanders}.

\begin{problem}
Determine the mixing time of the Markov chain from \cref{newmainthm} and the meanders Markov chain from \cite{meanders}.
\end{problem}

Note that a trivial upper bound for the mixing time of the Markov chain on $\dtg_n$ is $O(|V(\dtg_n)|^2)$ \cite[Proposition 10.28]{markov}. We suspect that this bound can be improved significantly because the mixing time for a random walk on $\tg_n$ is polynomial in $n$~\cite{mcshinetetali,molloy}, while $|V(\dtg_n)|$ grows rapidly with respect to $n$~\cite{countingplanar}. 
The computer data given in \cref{tab:mixing} supports our suspicion that  $O(|V(\dtg_n)|^2)$ is not a useful upper bound for the mixing time of $\dtg_n$. 

\begin{table}[h!]
    \centering
    \begin{tabularx}{\linewidth}{|*{6}{>{\centering\arraybackslash}X|}}
    \hline 
        $n$ & $5$ & $6$ & $7$ & $8$ & 9  \\ \hline 
        $|V(\mathcal{D}_n)|$ & $10$ & $68$ & $546$ & $4872$ & $46782$ \\ \hline iterations & $3$ & $7$ & $14$ & $25$ & $39$\\ \hline 
        $\sigma_2$ & $0.5590...$ & $0.7287...$ & $0.8478...$ & $0.9512...$ & $0.9677...$ \\\hline 
    \end{tabularx}
    \caption{For each $n$, the number of iterations needed for the total variation distance from the uniform distribution to be smaller than $1/4$ regardless of the initial vertex chosen, and the second largest eigenvalue of the transition matrix.}
    \label{tab:mixing}
\end{table}

The Markov chains discussed in this paper allow for near-uniform generation of the corresponding elements. Another natural question is what properties a ``typical" element has. One surprising result of \cite{randomtanglegram} is that as $n\to\infty$, generating a tanglegram of size $n$ uniformly at random becomes similar to generating two rooted plane binary trees on $n$ leaves independently and uniformly at random. Konvalinka and Wagner used this result to establish asymptotic properties of tanglegrams.

\begin{problem}
Study properties of pairs of disjoint triangulations, meanders, and planar tanglegrams generated uniformly at random. In the case of planar tanglegrams, establish analogues of results in \cite{randomtanglegram}.
\end{problem}

Known enumerative results for meanders \cite[A005315]{oeis} and pairs of disjoint triangulations \cite{countingplanar} imply that these objects are not in bijection with each other. However, it is possible that other connections exist by applying certain bijections between Catalan objects. 

\begin{problem}
Explore connections between pairs of disjoint triangulations, meanders, and other pairs of Catalan objects satisfying some property.
\end{problem}

Finally, we pose a problem involving our construction of $\dtg_n$. Our techniques on $\dtg_n$ can be generalized much further in the broader language of polyhedra.

\begin{problem} \label{op:polytope}
 For a given simple polytope $P$, explore the graph on pairs 
\[\{(u,v): u, v \in V(P), \, u \text{ and } v \text{ are not contained in the same facet}\}\]
with edges constructed in a manner similar to in $\dtg_n$. Determine what analogues of \cref{mainthm1} hold and
 when this is the graph of an abstract polytope.
\end{problem}

\section*{Acknowledgements}

This work was completed in part at the 2022 Graduate Research Workshop in Combinatorics, which was supported in part by NSF grant \#1953985 and a generous award from the Combinatorics Foundation. The authors acknowledge helpful discussions with GRWC participants Jesse Campion Loth and Samuel Mohr. The authors also thank Sara Billey, Herman Chau, Jes\'{u}s De Loera, and the anonymous referee for valuable feedback on earlier drafts of this manuscript.

Alexander E. Black was supported by the NSF GRFP and NSF DMS-1818969. Kevin Liu was supported by NSF DMS-1764012. Michael Wigal was supported by NSF GRFP under Grant No. DGE-1650044. Mei Yin was supported by the University of Denver's Professional Research Opportunities for Faculty Fund 80369-145601.

\printbibliography

\end{document}